\documentclass[12pt,a4paper]{amsart}
\usepackage{amsmath}
\usepackage{amssymb}
\usepackage{amsthm}
\usepackage{amsfonts}
\usepackage{mathrsfs}
\usepackage{color}
\usepackage{hyperref}
\usepackage{cleveref}
\usepackage[margin=1.0in]{geometry}
\usepackage{esint}

\numberwithin{equation}{section}

\newtheorem{theorem}{Theorem}[section]
\newtheorem{lemma}[theorem]{Lemma}
\newtheorem{corollary}[theorem]{Corollary}
\newtheorem{proposition}[theorem]{Proposition}

\newtheorem{question}{Question}

\theoremstyle{definition}
\newtheorem{definition}[theorem]{Definition}
\newtheorem{example}[theorem]{Example}

\theoremstyle{remark}
\newtheorem{remark}[theorem]{Remark}

\newcommand{\db}{\bar{\partial}}
\newcommand{\dd}{\partial}
\newcommand{\n}{\nabla}

\newcommand{\rr}[1]{\mathrm{#1}}
\newcommand{\dif}{\rr{d}}

\newcommand{\Li}{\mathcal{L}}

\newcommand{\leqs}{\leqslant}
\newcommand{\geqs}{\geqslant}
\newcommand{\pdt}{\frac{\partial}{\partial t}}

\newcommand{\ddt}{\frac{\dif}{\dif t}}

\newcommand{\W}{\mathcal{W}}

\newcommand{\rean}{\mathbb{R}}

\newcommand{\kah}{\text{K\"{a}hler}}
\newcommand{\brs}[1]{\left| #1 \right|}

\DeclareMathOperator{\tr}{tr}

\DeclareMathOperator{\Ric}{Ric}
\DeclareMathOperator{\diver}{div}
\DeclareMathOperator{\id}{id}

\DeclareMathOperator{\image}{im}

\title{On the shrinking solitons of generalized ricci flow}
\author{Xilun Li, Yanan Ye}
\address{}
\email{\href{mailto:lxl28@stu.pku.edu.cn}{lxl28@stu.pku.edu.cn}}
\email{\href{mailto:yeyanan@outlook.com}{yeyanan@outlook.com}}
\date{}

\begin{document}
\maketitle
\begin{abstract}
	We show that every gradient shrinking soliton of the generalized Ricci flow on compact manifold is a Ricci soliton.
 And we prove that the  pluriclosed soliton is gradient \kah-Ricci soliton under a broad cohomological condition.
 Moreover, we construct the first example of non-trivial shrinking generalized soliton, which can serve as a singularity model of the generalized Ricci flow.
\end{abstract}

\tableofcontents

\section{Introduction}

The generalized Ricci flow studied systematically by Streets and Tian\cite{MR2426247,MR3110582} is one of the most natural generalizations of the Ricci flow. 
To our knowledge, it appeared first in the mathematical literature in \cite{MR2426247}, originating from the research on the renormalization group flow in physics\cite{MR0819433,MR2214659}.
Formally speaking, a one-parameter family of Riemannian metrics $g(t)$ and a one-parameter family of closed 3-forms $H(t)$ is a solution to the generalized Ricci flow on a manifold $M^n$ if
\begin{align*}
    \partial_{t}g=-2\Ric+\frac12 H^2
    ,\quad
    \partial_{t} H=-\Delta_{d}H.
\end{align*}
Here $H^2(X,Y)=g(i_XH,i_YH)$ is a non-negative definite tensor and $\Delta_{d}=dd^*+d^*d$ is the Hodge Laplacian induced by time-dependent metric $g(t)$.
It is easy to see that when $H$ equals 0, it reduces to the classical Ricci flow.
In complex geometry, it is common to choose $H=d^c\omega$, where $\omega$ is the K{\"a}hler form associated with $g$. 
And in this case, it becomes the pluriclosed flow introduced by Streets and Tian\cite{MR2673720}, which plays a significant role in the study of non-K{\"a}hler complex manifolds (see e.g. \cite{MR4181011}).

Similar to the Ricci flow, self-similar solutions play a crucial role in the study of geometric flows. For the generalized Ricci flow, a generalized soliton is defined as the solution to
\begin{align*}
    2\Ric-\frac12 H^2=\lambda g-\mathcal{L}_{X}g
    ,\quad
    \Delta_{d}H=\lambda H-\mathcal{L}_{X}H
\end{align*}
for some real number $\lambda$ and vector field $X$.
It is referred to as expanding, steady, and shrinking generalized solitons when
$\lambda<0,=0,>0$, respectively.
And it is called gradient generalized soliton, if $X=\nabla f$ for some smooth function. 
Similarly, we can define the pluriclosed solitons in a completely analogous manner.
To elucidate the relationship between the generalized soliton and pluriclosed soliton, we prove that if a pluriclosed metric is conformally equivalent to a balanced metric, then it is K{\"a}hler (see Lemma \ref{balancepluriclosed}). This generalizes some results\cite{MR4598934} on the Fino-Vezzoni conjecture\cite{MR3327047}.
We refer the readers to Section \ref{sec-pluriclosedsoliton} for details.

Furthermore, using the classification of compact complex surfaces, Streets\cite{MR4023384} proves that non-trivial (i.e., $H\neq0$) compact steady pluriclosed solitons only exist on the Hopf surface, and he provides a one parameter family of examples.
Later, Streets and Ustinovskiy\cite{cpam} gave a simpler construction, and moreover showed that these are generalized \kah. Finally \cite{MR4619595} proved that these examples are the all compact steady generalized K{\"a}hler solitons in dimension two.

For expanding solitons, similar discussions to those of the Ricci flow\cite{MR3110582,MR4284898} can be adopted to prove that all expanding solitons on compact manifolds are Ricci solitons.
Streets\cite{MR4023384} proved that shrinking pluriclosed solitons on compact K{\"a}hler manifolds are K{\"a}hler Ricci solitons. Unfortunately, due to the lack of monotonicity of the generalized version of the $\W$ functional\cite{MR4284898}, we know very little about the shrinking generalized soliton in general.
In this paper, we will prove the triviality of shrinking generalized soliton under certain assumptions.

\begin{theorem}[c.f. Theorem \ref{noshrinkingsoliton}]
Every compact gradient shrinking generalized soliton is a Ricci soliton.
\end{theorem}

For compact pluriclosed solitons, we generalize the result of Streets in \cite{MR4023384}.

\begin{theorem}[c.f. Theorem \ref{thm-generalizedsoliton}]
    Suppose $(M^{2n},J,\omega,X)$ is a pluriclosed soliton on a compact complex manifold with $[\dd\omega]=0\in H^{2,1}_{\db}(M)$. Then $\omega$ is a gradient K\"{a}hler-Ricci soliton.
\end{theorem}

\begin{remark}
Note that the condition $[\dd\omega]=0$ is quite general.
It always holds on manifolds admitting the $\partial\bar{\partial}$-lemma.
Particularly, It always holds on K{\"a}hler manifolds.
Moreover, this condition is also satisfied for Hermitian-symplectic manifolds.
\end{remark}

This paper also provides examples of non-compact shrinking generalized solitons. Moreover, we also provide some descriptions for non-gradient compact shrinking solitons (see Section \ref{sec-noncompactexamples}).
For example, non-trivial compact generalized solitons can not admit any torsion-bounding function (Proposition \ref{prop-soliton_cannot_torbounding}).

The organization of this paper is as follows: In Section \ref{sec-preliminary}, we review the relevant concepts of the generalized Ricci flow and generalized solitons, along with some properties of the shrinking entropy; In Section \ref{sec-compactshrinking}, we present the proof of Theorem \ref{noshrinkingsoliton} and give a criterion for the triviality (i.e., $H=0$) of complete solitons; In Section \ref{sec-noncompactexamples}, we provide examples of non-compact generalized solitons in three dimensions on the cylinder and on $\mathbb{R}^3$, one of which is non-collapsing and can not admit any torsion-bounding function. Moreover, we establish that non-trivial compact shrinking generalized solitons can not admit any torsion-bounding function; In Section \ref{sec-pluriclosedsoliton}, we delve into the relationship between generalized solitons and pluriclosed solitons, and give the proof of Theorem \ref{thm-generalizedsoliton}.
\newline

\textbf{Acknowledgements} We want to express our sincere gratitude to our advisor, Professor Gang Tian, for his
helpful suggestions and patient guidance. We are also grateful to Shengxuan Zhou and Professor Jeffrey D. Streets for their valuable comments. Authors are supported by National Key R\&D Program of China 2020YFA0712800.

\section{Preliminary}\label{sec-preliminary}
We first review some fundamental concepts and properties of the generalized Ricci flow and solitons.
\subsection{Generalized flow and generalized soliton}
Given a one parameter family of Riemannian metrics $g(t)$ and a one parameter family of closed three form $H(t)$.
We say $(M^n,g(t),H(t))$ is a generalized Ricci flow on manifold $M^n$ if it satisfies
\begin{align*}
\pdt g=-2\Ric+\frac12H^2
,\quad
\pdt H=-\Delta_{d}H,
\end{align*}
where $H^2$ is a non-negative symmetric tensor defined by $H^2(X,Y)=g(i_{X}H,i_{Y}H)$ and $\Delta_{d}=dd^*+d^*d$.

\begin{definition}
   $(M^n,g,H)$ is called a \emph{generalized soliton} if there exists a real vector field $X$ a real constant $\lambda$ such that
   \begin{align*}
       2\Ric_g-\frac12H^2&=\lambda g-\Li_Xg,\\
       \Delta_dH&=\lambda H-\Li_XH,
   \end{align*}
\end{definition}
A soliton is called \emph{shrinking}, \emph{steady} or \emph{expanding} if $\lambda>0$, $=0$ or $<0$ respectively. A soliton is called \emph{gradient} if $X=\nabla f$ for some function $f\in C^{\infty}(M,\mathbb{R})$. 
Notice that we can always assume $\lambda \in \{-1,0,1\}$ after scaling.

In \cite{MR3110582}, Streets and Tian show that the generalized Ricci flow is a gradient flow in a sense similar to the Ricci flow.
And they give the construction of $\mathcal{F}$ functional for the generalized Ricci flow.
As a direct application, they prove that on compact manifold, generalized steady solitons are gradient and generalized expanding solitons are expanding Ricci solitons (i.e., $H=0$).
And Streets and Garcia-Fernandez\cite{MR4284898} show that in the generalized flow, the analogue of the Ricci flow $\W$-functional is not monotonic, which is one of the reasons for the limited understanding of shrinking solitons.

\subsection{Shrinking entropy}
In this section, we review some fundamental concepts and properties of the shrinking entropy introduced by Streets and Garcia-Fernandez in \cite{MR4284898}. 
For the convenience of later use in this paper, we list some proofs in the appendix.
More detailed information can be found in \cite[Chapter 6]{MR4284898}.

\begin{definition}[\cite{MR4284898}, Definition 6.22]
	Given a Riemannian manifold $(M^n,g,H)$ with a closed three form $H$, a smooth function $f\in C^{\infty}$ and a constant $\tau$.
	We can define the  \emph{shrinking entropy} by
\begin{align*} 
    \W_-(g,H,f,\tau)  := \int_M \left[ \tau \left( |\nabla f|^2 + R - 
\frac{1}{12}|H|^2 \right) + f - n \right] u dV,
\end{align*}
in which $u = (4\pi\tau)^{-\frac{n}{2}}e^{-f}$.
\end{definition}

Similar to the Ricci flow, we can also consider the conjugate heat equation for the generalized Ricci flow.

\begin{definition}
    Let $(M^n, g_t, H_t)$ be a solution to generalized Ricci flow.  A 
one-parameter family $u_t \in C^{\infty}(M)$ is a solution to the 
\emph{conjugate heat equation} if
\begin{align*} 
\left(\pdt + \Delta_{g_t} \right) u = R u - \frac{1}{4}|H|^2 u.
\end{align*}
\end{definition}
For convenience, we will use $\square$ to denote the \emph{heat operator} and use $\square^*$ to denote the \emph{conjugate heat operator}. 
In other words, we have
\begin{align*}
    \square := \pdt - \Delta_{g_t},\ \ \square^* := - \pdt - \Delta_{g_t} + R - \frac{1}{4} |H|^2.
\end{align*}
And for $w,\  v\in C^{\infty}$, we have the following formula 
\begin{align}\label{eq-conjugation-relationship}
\ddt \int_M wv dV_{g(t)}=\int_M \square w\cdot vdV_{g(t)}-\int_M  w \square^*vdV_{g(t)},
\end{align}
along the generalized Ricci flow.

To compute the derivative of the shrinking entropy, we need the following lemma. 
The proof of this lemma is crucial for the subsequent discussion, so we also list it in the appendix for easy reference.

From now on, we will always use $u$ to denote the weighted function $u= (4\pi\tau)^{-\frac{n}{2}}e^{-f}$ associated to $f$ and $\tau$.

\begin{lemma}\label{lem-point-monotonicity}
Let $(M^n, g_t, H_t)$ be a solution to generalized Ricci flow on a compact manifold. 
We define
\begin{align*}
  v = \left[ (T - t) \left( 2 \Delta f - |\nabla f|^2 + R - \frac{1}{12}|H|^2 \right) + f - n \right] u.
 \end{align*}
If $u_t $ is a solution to the conjugate heat equation, then
\begin{align*}
    \square^*v=-\left(2 \tau \brs{\Ric - \frac{1}{4} H^2 + \nabla^2 f - \frac{g}{2 \tau}}^2 + \frac{\tau}{2}
\brs{d^* H + i_{\nabla f} H}^2  - \frac{1}{6} \brs{H}^2 \right)u.
\end{align*}
\end{lemma}

Applying Lemma \ref{lem-point-monotonicity} and the formula \eqref{eq-conjugation-relationship}, we can obtain the derivative formula of shrinking entropy along a generalized Ricci flow.

\begin{proposition}[\cite{MR4284898}, Proposition 6.26, Theorem 6.20]\label{entropymonotone}
    Let $(M^n, g_t, H_t)$ be a solution to generalized Ricci flow on a compact manifold.
    Suppose $\tau = T - t$ for some time $T$ and $u_t $ is a solution to the conjugate heat equation.
    Then
\begin{align*}
\ddt  \W_-(g_t,H_t,f_t,\tau_t)
=\ \int_M \Bigg[ 2 \tau \left|\Ric - \frac{1}{4} H^2
+ \nabla^2 f - \frac{1}{2 \tau}
g\right|^2 
+\frac{\tau}{2}\left|d^* H 
+ i_{\nabla f} H\right|^2 
- \frac{1}{6} |H|^2
\Bigg]
udV.
\end{align*}
\end{proposition}
The functional $\W_-$ here is not truly an ``entropy'', as it is not monotonic in general.
Streets and Garcia-Fernandez \cite{MR4284898} show that if there exists a torsion-bounding function, which is a strictly positive function $\varphi$ satisfies
\begin{align*}
\square\varphi\leq -|H|^2,
\end{align*}
then we can adjust the definition of $\W_-$ by subtracting the integral of $\varphi$ to obtain a monotonic functional.
And they proved that under these assumptions, the generalized Ricci flow is non-collapsing.
The torsion-bounding function is widely available, for example in the Hermitian-symplectic case (Theorem \ref{thm-generalizedsoliton}), and in some special cases on elliptic bundles over Riemannian surfaces \cite{MR4348696}.
However, the existence of a control function is not a necessary condition for non-collapse of the flow. We will provide an example later in the paper of a non-collapsing solution without a control function.
And we will prove that any compact shrinking soliton with $H\neq0$ can not admit any torsion-bounding function.

\section{Compact shrinking generalized solitons}\label{sec-compactshrinking}
In this section, we discuss the shrinking generalized solitons on compact manifolds.
In particular, we will prove that compact gradient shrinking solitons are Ricci solitons.
Without loss of generality, we assume the soliton constant $\lambda=1$.

Let $(M^n,g_0,H_0,X_0)$ be a generalized soliton.
We can find a family of diffeomorphism $\phi_t$ generated by $X_t:=(1-\lambda t)^{-1}X$ with $\phi_0=\id$.
Then $g_t=(1-\lambda t)\phi_t^*g_0$ and $H_t=(1-\lambda t)\phi_t^* H_0$ is a solution to the generalized Ricci flow with initial data $(g_0,H_0)$.
Actually, we have
\begin{align*}
\pdt g_t&=-\lambda\phi_t^*g_0+(1-\lambda t)\phi_t^*(\mathcal{L}_{X_t}g_0)
=\phi_t^*(-\lambda g_0+\mathcal{L}_{X_0}g_0)
\\
&=\phi_t^*(-2\Ric(g_0)+\frac12H_0^2)
=-2\Ric(g_t)+\frac12H_t^2.
\end{align*}
The last equality holds because
\begin{align*}
(H_t^2)_{ij}=g_t^{pq}g_t^{kl}(H_{t})_{ipk}(H_{t})_{jql}
=(\phi_t^*g_0)^{pq}(\phi_t^*g_0)^{kl}(\phi_t^* H_0)_{ipk}(\phi_t^* H_0)_{jql}
=(\phi_t^*H_0^2)_{ij}.
\end{align*}
Similarly, we have
\begin{align*}
\pdt H_t&=-\lambda\phi_t^*H_0+(1-\lambda t)\phi_t^*(\mathcal{L}_{X_t}H_0)
=\phi_t^*(-\lambda H_0+\mathcal{L}_{X_0}H_0)
\\
&=\phi_t^*(-\Delta^{g_0}_d H_0)
=-\Delta^{g_t}_d H_t.
\end{align*}
For gradient shrinking solitons, we have additional properties.

\begin{proposition}\label{solitonheateqn}
    Let $(M^n,g_0,H_0,f_0)$ be a gradient shrinking generalized soliton. 
    Set $f_t=\phi_t^*f_0$, where $\phi_t$ is a family of diffeomorphism $\phi_t$ generated by $X_t:=(1-\lambda t)^{-1}\nabla_0 f_0$ with $\phi_0=\id$.
    Then we have
    \begin{align*}
            \pdt f=-\Delta f+|\nabla f|^2-R+\frac14|H|^2+\frac{n}{2(1-t)}.
    \end{align*}
    In particular, $u_t:=(4\pi(1-t))^{-\frac{n}{2}}e^{-f_t}$ is a solution to the conjugate heat equation.
\end{proposition}
\begin{proof}
Since $f=\phi_t^*f_0$, we have
    \begin{align*}
        \pdt f=\phi_t^*\Li_{X_t}f_0=\phi_t^*\left(\frac{1}{1-t}\left|\nabla_0f_0\right|^2_{g_0}\right)=|\nabla f|^2.
    \end{align*}
    Contracting the soliton equation
\begin{align*}
2\Ric_0-\frac12H_0^2=g_0-2\nabla_0^2f_0,
\end{align*}
we obtain
    \begin{align*}
        R_0-\frac14|H_0|^2_{g_0}=\frac{n}{2}-\Delta_0 f_0.
    \end{align*}
    By direct computation,
    \begin{align*}
        \Delta f=&\frac{1}{1-t}\phi_t^*\Delta_0 f_0=\frac{1}{1-t}\phi_t^*\left(-R_0+\frac14|H_0|^2_{g_0}+\frac{n}{2}\right)=-R+\frac14|H|^2+\frac{n}{2(1-t)}.
    \end{align*}
    Combining above, we have
    \begin{align*}
        \pdt f=-\Delta f+|\nabla f|^2-R+\frac14|H|^2+\frac{n}{2(1-t)}.
    \end{align*}
    Then we obtain $\square^*u=0$ from the computation in the proof of Proposition \ref{lem-point-monotonicity} which is listed in the Appendix.
\end{proof}

Next, we present an observation that plays a key role in the proof of Theorem \ref{noshrinkingsoliton}.
\begin{proposition}\label{lottadjoint}
    Suppose $(M^n,g)$ be a compact manifold, $f\in C^\infty(M)$, $H\in \Lambda^3(M)$, $dH=0$. 
    We have
    \begin{align}\label{eq-lott-harmonic}
	\Delta_d H+\mathcal{L}_{\nabla f}H=d(e^{f}d^*(e^{-f}H))
	\end{align}
	and
    \begin{align}\label{eq-integral-lott-harmonic}
        \int_M\left|d^*H+i_{\nabla f}H\right|^2 e^{-f} dV=\int_M \left<\Delta_d H+\Li_{\nabla f}H,H\right> e^{-f} dV.
    \end{align}
\end{proposition}
\begin{proof}
	Firstly, we claim that
	\begin{align}\label{eq-suobing}
	i_{\nabla f}H=*(df\wedge* H).
	\end{align}
	Since both sides of the equation are bilinear, we only need to consider the following two cases under the standard orthonormal basis $\{e_1,\cdots,e_n\}$ and the dual basis$\{e^1,\cdots,e^n\}$.
	The first case is $\nabla f=e_1$ and $H=e^1\wedge e^2\wedge e^3$.
	We have
\begin{align*}
*(e^1\wedge*H)=*(e^1\wedge e^4\wedge\cdots\wedge e^n)=e^2\wedge e^3
=i_{e_1}e^1\wedge e^2\wedge e^3
\end{align*}
The second case is $\nabla f=e_1$ and $H=e^2\wedge e^3\wedge e^4$.
\begin{align*}
*(e^1\wedge*H)=-*(e^1\wedge e^1\wedge e^5\wedge\cdots\wedge e^n)=0=i_{e_1}e^2\wedge e^3\wedge e^4
\end{align*}
Then applying equation \eqref{eq-suobing}, we get
\begin{align*}
d^*H+i_{\nabla f}H
&=-*d*H+*(df\wedge *H)
=-*e^{f}e^{-f}(d*H-df\wedge *H)
\\
&=-e^{f}*d(e^{-f}*H)=e^{f}d^*(e^{-f}H)
\end{align*}
Applying the operator $d$ to both sides of the above equation, we obtain Equation \eqref{eq-lott-harmonic}.
\begin{align*}
	\Delta_d H+\mathcal{L}_{\nabla f}H=d(e^{f}d^*(e^{-f}H))
\end{align*}
To prove Equation \eqref{eq-integral-lott-harmonic}, it suffices to note that 
\begin{align*}
&\int_M\left|d^*H+i_{\nabla f}H\right|^2 e^{-f} dV
=\left(e^{f}d^*(e^{-f}H),e^{f}d^*(e^{-f}H)e^{-f}\right)_2
\\
&=\left(d(e^{f}d^*(e^{-f}H)),e^{-f}H\right)_2
=\int_{M}\left<\Delta_d H+\Li_{\nabla f}H,H\right> e^{-f} dV,
\end{align*}
in which $(\cdot,\cdot)_2$ is the $L^2$ inner product induced by metric.
\end{proof}

\begin{theorem}\label{noshrinkingsoliton}
Every compact gradient shrinking generalized soliton is a Ricci soliton.
\end{theorem}
\begin{proof}
	Given a shrinking generalized soliton $(M^n,g_0,H_0,f_0)$ on a compact manifold $M^n$.
    Let $(M^n,g_t,H_t,f_t)$ be the generalized Ricci flow induced by the soliton.
    By Proposition \ref{solitonheateqn}, $u_t:=(4\pi(1-t))^{-\frac{n}{2}}e^{-f_t}$ solves the conjugate heat equation with initial condition $f(0)=f_0$. Taking $T=1$ in the Proposition \ref{entropymonotone},  we have
    \begin{align*}
        \pdt & \W_-(g_t,H_t,f_t,\tau_t)\\
=&\ \int_M \left[ 2 \tau \left|\Ric - \frac{1}{4} H^2
+ \nabla^2 f - \frac{1}{2 \tau}
g\right|^2 +\frac{\tau}{2}\left|d^* H + i_{\nabla f} H\right|^2 - \frac{1}{6} |H|^2
\right]
(4 \pi \tau)^{-\frac{n}{2}} e^{-f} dV.
    \end{align*}
 Since $g_t=\tau\phi_t^*g_0$, $H_t=\tau\phi_t^*H_0$, $f_t=\phi_t^*f_0$ and $u_t dV_t=\phi_t^*(u_0dV_0)$, for any $t\in[0,1)$, we have
\begin{align*}
    \W_-(g_t,H_t,f_t,\tau)&= \int_M \left[ \tau \left( |\nabla f|^2 + R - 
\frac{1}{12}|H|^2 \right) + f - n \right] u dV\\
&= \int_M \phi_t^*\left[ \left( |\nabla_0 f_0|^2_{g_0} + R_0 - 
\frac{1}{12}|H_0|^2_{g_0}  + f_0 - n \right) u_0 dV_0\right]\\
&=\W_-(g_0,H_0,f_0,1).
\end{align*}
For $t_0=0$,
\begin{align*}
    \int_M \left[ 2 \left|\Ric_0 - \frac{1}{4} H^2_0
+ \nabla^2_0 f_0 - \frac{1}{2}
g_0\right|^2 +\frac{1}{2}\left|d^* H_0 + i_{\nabla_0 f_0} H_0\right|^2 - \frac{1}{6} |H_0|^2
\right]
(4 \pi)^{-\frac{n}{2}} e^{-f_0} dV=0.
\end{align*}
By Proposition \ref{lottadjoint},
\begin{align*}
    \int_M\left|d^*H+i_{\nabla f}H\right|^2 e^{-f} dV=\int_M \left<\Delta_d H+\Li_{\nabla f}H,H\right> e^{-f} dV=\int_M|H|^2e^{-f} dV.
\end{align*}
From the soliton equation
\begin{align*}
\Ric_0 - \frac{1}{4} H_0^2
+ \nabla^2_0 f_0 - \frac{1}{2}g_0=0,
\end{align*}
we get
\begin{align*}
    \frac13\int_M|H|^2e^{-f}dV=0.
\end{align*}
This implies that $H\equiv0$, which means it is a Ricci soliton.
\end{proof}

For the complete shrinking generalized soliton, we also have some criterion to determine when they are Ricci solitons.
\begin{theorem}
    Suppose $(M^n,g,H,f)$ is a complete gradient shrinking generalized soliton with bounded $|\nabla f|_g$ and $\Delta f$, then $M^n$ is compact. In particular, $H=0$.
\end{theorem}
\begin{proof}
	The idea behind this proof is inspired by an observation of the Bakry-\'{E}mery-Ricci tensor by Lott \cite{MR2016700}.
	 Consider $M^n\times S^q$ with the warped product 
	 \begin{align*}
g=g_M+e^{-2f/q}g_{S^q},
	\end{align*} where $q$ is a positive integer to be determined later.
	  Let $p:M^n\times S^q\to M$ be the projection and $\overline{X}$ be  the horizontal lift to $M^n\times S^q$ of a vector field $X$ on $M$ and let $\overline{U}$ be a vertical vector ﬁeld on $M^n\times S^q$. Then
    \begin{align*}
        \Ric_g(\overline{X},\overline{X})&=p^*\left[\widetilde{\Ric}_q(M,g_M,f)(X,X)\right],\\
        \Ric_g(\overline{X},\overline{U})&=0,\\
        \Ric_g(\overline{U},\overline{U})&=\Ric_{S^q}(\overline{U},\overline{U})+\frac{1}{q}|\overline{U}|_g^2\cdot p^*\left(\Delta f-|\nabla f|^2_M\right).
    \end{align*}
	
    Since $|\nabla f|_g\leqs C_0$ and $-\Delta f\leqs C_0$, we have
    \begin{align*}
	\widetilde{\Ric}_\infty(M,g,f)=\Ric_g+\nabla^2 f=g+\frac14H^2\geqs g.
\end{align*}
And we can take $q$ sufficiently large such that
\begin{align*}
\widetilde{\Ric}_q=\widetilde{\Ric}_\infty-\frac{1}{q}df\otimes df\geqs\varepsilon g>0.
\end{align*}
Then $M^n\times S^q$ is compact by Myers' theorem, since there exists a metric whose Ricci curvature has uniformly positive lower bound.
So $M$ is compact and $H=0$ by Theorem \ref{noshrinkingsoliton}.
\end{proof}

\begin{remark}
In subsequent sections, we provide an example (Proposition \ref{eg-S2R}) of a non-compact gradient shrinking generalized soliton where $\nabla f$ is unbounded but $\Delta f$ is bounded. This demonstrates that the unboundedness of $\Delta f$ is not a necessary condition for the existence of non-compact gradient shrinking generalized solitons.
\end{remark}

\begin{corollary}
     Suppose $(M^n,g,H)$ is complete and shrinking generalized Einstein, then $M$ is compact and $H=0$.
\end{corollary}

\section{Non-trivial shrinking generalized solitons}\label{sec-noncompactexamples}

In this section, we first consider examples of non-trivial (i.e., $H\neq0$) shrinking solitons with three-dimensional rotational symmetry. Then, we point out that the example on the cylinder can serve as a singularity model of the generalized Ricci flow. To our knowledge, this is the first known non-trivial example of the singularity model. Furthermore, this example illustrates that the existence of a torsion-bounding function is not a necessary condition for non-collapse of generalized Ricci flows. In fact, we can also prove that shrinking generalized solitons on compact manifolds admit torsion-bounding functions if and only if they are trivial.

\subsection{Examples in dimension three}

Similar to Bryant's work\cite{Bryant} (also see \cite[(3.5)]{2401.05028}), we know that a warped product metric $(\rean\times_{\phi} S^2, dr^2+\phi^2(r)g_{S^2})$ is a $SO(3)$-invariant generalized gradient soliton if and only if the following system of ODEs holds:
\begin{equation}\label{odes-soliton}
    \begin{aligned}
        1-(\phi')^2-\phi\phi''&=\lambda\phi^2-\phi\phi'f'+\frac12h^2\phi^2,\\
-2\phi\phi''&=(\lambda-f'')\phi^2+\frac12h^2\phi^2,\\
(\phi^2h')'&=-2\lambda h\phi^2+(f'h\phi^2)',
    \end{aligned}
\end{equation}
where $\nabla f=f'(r)\partial_r$ is the soliton vector field, $H=h(r)dV$ and $\Ric(g_{S^2})=g_{S^2}$.

Under the assumption that $h$ is a non-zero constant, the third equation becomes
\begin{align*}
    2\lambda \phi^2=f''\phi^2+2f'\phi\phi'.
\end{align*}
By linearly combining the equations in \eqref{odes-soliton}, we obtain
\begin{align*}
    2-2(\phi')^2-4\phi\phi''&=(\lambda+\frac32h^2)\phi^2.
\end{align*}
For convenience, we can assume that $\phi$ satisfies
\begin{align}\label{Eq-R3-phi}
    \phi^2+(\phi')^2+2\phi\phi''=1
\end{align}
after scaling.
More precisely, set $\phi(r)=a\tilde{\phi}(br)$.
Then we have
\begin{align*}
    \phi'=ab\tilde{\phi}'
    ,\quad
    \phi''=ab^2\tilde{\phi}''
\end{align*}
and
\begin{align*}
    2a^2b^2\tilde{\phi}^2+4a^2b^2\tilde{\phi}\tilde{\phi}''
    +(\lambda+\frac32h^2)a^2\tilde{\phi}=2
\end{align*}
We can choose $a$,$b$ such that $a^2b^2=1$ and $a^2(\lambda+\frac32h^2)=2$.

First, we consider the Cauchy problem of equation \eqref{Eq-R3-phi} with initial condition $\phi(0)=0$ and $\phi'(0)>0$, which corresponds to rotationally symmetric solutions on $\mathbb{R}^3$.

\begin{proposition}
    Let $(\rean^3,g,H,f)$ be a $SO(3)$-invariant gradient shrinking generalized soliton with $d^*H=0$. 
    Then it's a Ricci soliton.
\end{proposition}
\begin{proof}
    First of all, notice that in dimension three, $d^*H=0$ is equivalent to that $h$ is a constant.
    Assume $h\neq0$ and let $\phi(r)$ be a solution to \eqref{Eq-R3-phi}. By \cite{MR3469435}, the smoothness of $g$ implies $\phi(0)=0$, $\phi'(0)>0$. After a shift of $r$, we can assume $\phi(0)>0$, $\phi'(0)>0$. Set $u:=\phi^{\frac32}$, then        
    $u'=\frac32\phi^{\frac12}\phi'$ with initial condition 
    $u(0)>0$, $u'(0)=\varepsilon>0$.
    By direct computation,
    \begin{align*}
        u''&=\frac34\phi^{-\frac12}(\phi')^2+\frac32\phi^{\frac12}\phi''=\frac34\phi^{-\frac12}((\phi')^2+2\phi\phi'')\\
        &=\frac34\phi^{-\frac12}(1-\phi^2)\\
        &=\frac{3}{4}(u^{-\frac13}-u).
    \end{align*}
    Taking $p:=u'$, it becomes
    \begin{align*}
        u''=\frac{dp}{dr}=\frac{dp}{du}\frac{du}{dr}=p\frac{dp}{du}=\frac{3}{4}(u^{-\frac13}-u).
    \end{align*}
    or equivalently,
    \begin{align*}
        3u^2+4p^2=9u^{\frac23}.
    \end{align*}
    Next, we show that there exists
    $r>0$ such that $\phi=0$, which contradicts the definition of the warp product metric. Hence $h$ must vanish.

    Step 1: There exists $r_1>0$ such that $u(r_1)=1$ and $u'(r_1)>0$. If not, we have $u<1$ all the time. So we obtain $u''>0$, as a consequence we get $u'(r)>u'(0)=\varepsilon$ and then $u(r)>u(0)+\varepsilon r$. Thus $u(r)>1$ for sufficiently large $r$. It's a contradiction.

    Step 2: There exists $r_2>r_1$ such that $u(r_2)>1$ and $u'(r_2)=0$. If not, we have $u'>0$, hence $u$ is strictly increasing. Then there exists $\delta_1>0$ such that for any $r>\tilde{r}_1:=r_1+1$, we have $u(r)>1+\delta_1$ and $u''(r)<-\delta_1$. Thus $u'(r)<u'(\tilde{r}_1)-\delta_1(r-\tilde{r}_1)$. This leads to a contradiction for sufficiently large $r$.

    Step 3: There exists $r_3>r_2$ such that $u(r_3)=1$ and $u'(r_3)<0$. If not, we have $u>1$, and then $u''<0$.
    So there exists $\delta_2>0$ such that for any $r>\tilde{r}_2:=r_2+1$, we have $u'(r)<-\delta_2$. Thus $u(r)<u(\tilde{r}_2)-\delta_2(r-\tilde{r}_2)$.
    This leads to a contradiction for sufficiently large $r$.

    Step 4: There exists $r_4>r_3$ such that $u(r_4)=0$, $u'(r_4)=0$. If not, we have $u'<0$, hence $u$ is strictly decreasing. Then there exists $\delta_3>0$ such that for any $r>\tilde{r}_3:=r_3+1$, $u(r)<1-\delta_3$, $u''(r)>\delta_3$. Thus $u'(r)>u'(\tilde{r}_3)+\delta_3(r-\tilde{r}_3)$, which provides a contradiction when $r$ is large enough.
\end{proof}

Then we provide an example of non-trivial (i.e., $H\neq0$) shrinking generalized soliton on cylinder, which corresponds to the Cauchy problem of equation \eqref{Eq-R3-phi} with initial condition $\phi(0)>0$.

\begin{example}\label{eg-S2R}
    There exists non-trivial shrinking generalized soliton on $S^2\times\rean$.
\end{example}
\begin{proof}
    It is easy to check that $\lambda=\frac12$, $\phi=1$, $h=1$, $f=\frac12r^2$ is a solution to \eqref{odes-soliton} with initial condition $\phi(0)>0$.
\end{proof}

\subsection{Non-trivial singularity model}
The singularity models of the Ricci flow played a crucial role in Perelman's resolution of the Poincar{\'e} conjecture\cite{math/0211159}. For the generalized Ricci flow, due to the lack of monotonicity of the generalized $\mathcal{W}$ functional, it is not even known in general whether it collapses or not. Streets and Garcia-Fernandez\cite{MR4284898} introduced a condition called admitting torsion-bounding function and proved that under this assumption, the generalized Ricci flow are non-collapsing. More precisely, for a solution to the generalized Ricci flow $(g(t),H(t))$, $\psi$ is a torsion-bounding function if $\psi\geq0$ and satisfies
\begin{align*}
    \square\psi\leqs-|H|^2.
\end{align*}
Streets\cite{MR4348696} proved that in the context of the  Ricci-Yang-Mills flow on $S^2\times T^k$, which is a special symmetric reduction of generalized Ricci flow, the torsion-bounding function always exists, ensuring that the flow is non-collapsing, and in such cases, generalized singularity models are trivial (i.e, $H=0$).

We point out that the soliton constructed in Example \ref{eg-S2R} will occur naturally as a singularity model of compact generalized Ricci flow and prove that it can not admit any torsion-bounding function.

We analyze the generalized Ricci flow on $(S^2\times S^1,g_{S^2}+f^2dr^2)$ with initial data
\begin{align*}
    \Ric_{g_{S^2}(0)}=\frac12g_{S^2}(0), H(0)=h_0 dV_{g_0}, \nabla f_0=\nabla h_0=0.
\end{align*}
Then $g_{S^2}(t)=\lambda(t)g_{S^2}(0)$, $H(t)=h(t)dV_{g(t)}$, $\nabla\lambda(t)=\nabla f(t)=\nabla h(t)=0$. Hence the flow equation becomes an ODE system:
\begin{align*}
	&\lambda'=-1+\lambda h^2,\\
	&h'=\frac{h}{\lambda}-\frac32h^3,\\
	&f'=\frac12h^2f,\\
	&\lambda(0)=1,\ h(0)=h_0,\ f(0)=f_0.
\end{align*}
\begin{proposition}\label{eg-S2S1}
	Under the settings above, we have
	\begin{enumerate}
		\item If $h_0=0$, then the solution is
  \begin{align*}
      (\lambda(t),h^2(t),f(t))=(1-t,0,f_0)
      ,\quad
      t\in [0,1);
  \end{align*}
		\item If $h_0^2=\frac12$, then the solution is
  \begin{align*}
      (\lambda(t),h^2(t),f(t))=\left(1-\frac12t,\frac{1}{2-t},
      \frac{f_0}{\sqrt{1-\frac12t}}\right)
      ,\quad
      t\in [0,2);
  \end{align*}
		\item If $h_0\neq0$, then the flow exists for $t\in [0,T)$, $T<\infty$. Moreover, for the blowup sequence $(g_i(t),H_i(t)$
we have $(S^2\times S^1,g_i(t),H_i(t))\to (S^2\times\rean,g_\infty(t),H_\infty(t))$, where
\begin{align*}
    g_\infty(0)=g_{S^2}(0)+dr^2
    ,\quad
    h_\infty^2(0)=\frac12.
\end{align*}
In fact, $(S^2\times\rean,g_\infty,H_\infty)$ is exactly the shrinking generalized soliton constructed in Proposition \ref{eg-S2R} up to a scaling.
	\end{enumerate}
\end{proposition}
\begin{remark}
Here the blowup sequence is defined by
  \begin{align*}
g_i(t):=\lambda(t_i)^{-1}g(t_i+\lambda(t_i)t)
,\quad
H_i(t):=\lambda(t_i)^{-1}H(t_i+\lambda(t_i)t)
,\quad
t_i\to T^-.
\end{align*}
\end{remark}
\begin{proof}
	The first and second statement follow straightforward computation. For the third one, we compute
	\begin{align*}
		&(\lambda h^2)'=(-1+\lambda h^2)h^2+2\lambda h(\frac{h}{\lambda}-\frac32h^3)=h^2-2\lambda h^4,\\
		&(\lambda h f)'=(-1+\lambda h^2)h f+\lambda f(\frac{h}{\lambda}-\frac32 h^3)+\lambda h(\frac12 h^2f)=0.
	\end{align*}
	Firstly we will show that $T<\infty$ and $\lambda h^2(t)\to\frac12$ as $t\to T^-$. It's divided into two cases.
	
	\textbf{Case 1:} $0<h_0^2<\frac12$. Taking $u:=\frac12-\lambda h^2$, we have
 $u'=-2h^2u$, $u(0)>0$. 
    Assume the maximum existence time of ODEs is $T\leqslant\infty$, then
    \begin{align*}
        -2\int_0^t h^2(s)ds>-\infty
        ,\quad
        t<T.
    \end{align*}
  Hence $u>0$ and we can consider $\log u$.
  Notice that
  $(\log u)'=-2h^2\leqslant0$, which provides that $u$ is decreasing and $\lambda h^2$ is increasing. 
  Then we obtain $T<2$ since $\lambda'=-\frac12-u<-\frac12$ and $\lambda>0$ by definition.
  Notice that
  \begin{align*}
      0<h_0^2\leqslant\lambda h^2<\frac12
      ,\quad
      u(t)\to u_T\geqslant0
      ,\quad
      t\to T^-.
  \end{align*}
  If $u_T>0$, then
	\begin{align*}
		\log u_T-\log u(0)&=-2\int_0^T h^2(s)ds\leqslant-2\int^T_0\frac{h_0^2}{\lambda(s)}ds=-2\int^T_0\frac{h_0^2}{(s-T)\lambda'(\xi_s)}ds\\
		&\leqslant-2\int^T_0\frac{h_0^2}{T-s}ds=-\infty.
	\end{align*}
	 So we have $u_T=0$, namely, $\lambda h^2\to\frac12$.
	 
	 \textbf{Case 2:} $h_0^2>\frac12$. Taking $u=\lambda h^2-\frac12$, we have $u'=-2h^2u$. For the same reason, we have $u>0$ in $[0,T)$, $u$ and $\lambda h^2$ are decreasing.
  We prove by contradiction.
  
  Assume $T=+\infty$. We have $\lambda h^2\geqslant1$. This is because that
  if there exists $t_0>0$ such that $\lambda h^2(t_0)<1$, then
  \begin{align*}
      \lambda'=\lambda h^2-1\leqslant \lambda h^2(t_0)-1<0
      ,\quad
      t>t_0
  \end{align*}
  which provides $T<\infty$.
 By the monotonicity of $u$, we have
 $u(t)\to u_\infty\geqslant\frac12$ as $t\to\infty$
 and
 \begin{align*}
     2\int_0^\infty h^2(s)ds=\log u(0)-\log u_\infty<\infty.
 \end{align*}
 Then
 \begin{align*}
     \int_0^\infty \lambda^{-1}ds\leqslant\int_0^\infty h^2ds<\infty.
 \end{align*}
We will prove that this cannot happen under our assumption.
Notice $\lambda'=\lambda h^2-1\leqslant h_0^2-1=:a_0$.
Set $a_0=h_0^2-1$.

If $a_0\leqslant0$, we have $\lambda\leqslant1$ and then $\int_0^\infty\lambda^{-1}ds\geqslant+\infty$. 

If $a_0>0$, we have
	 \begin{align*}
	 	\infty>\int_0^\infty\lambda^{-1}dt\geqslant\int_0^\infty\frac{dt}{1+a_0t}=\left.\frac{1}{a_0}\log(1+a_0t)\right|_0^\infty=+\infty.
	 \end{align*}
	 Thus we have proven that $T<\infty$.
  Since $\lambda(t)\to 0$ as $t\to T^-$, there exists $\varepsilon_0>0$ such that $-1<\lambda'(t)\leqslant-\varepsilon_0$ in $[T-\varepsilon_0,T)$.
  Notice that $\frac12<\lambda h^2\leqslant h_0^2$ and the argument is similar to Case 1.
  
  The blowup sequence at $t=0$ is
  \begin{align*}
    g_i(0)=\lambda(t_i)^{-1}g(t_i)=g_{S^2}(0)+\lambda(t_i)^{-1}f^2(t_i)dr^2
  \end{align*}
We have	$\lambda^{-1}f^2=h_0^2f_0^2\lambda^{-3}h^{-2}$ since $\lambda h f=h_0f_0$.
Notice $\lambda^{-1}h^{-2}\to 2$, so $\lambda^{-1}f^2\to +\infty$. 
Then we obtain the smooth convergence $(S^2\times S^1,g_i(0))\to (S^2\times \rean, g_{\infty}(0))$ with $g_\infty(0)=g_{S^2(0)}+dr^2$. 
By direct computation,
\begin{align*}
h_i(0)dV_{g_i(0)}=H_i(0)=\lambda(t_i)^{-1}H(t_i)=\lambda(t_i)^{-1}h(t_i)dV_{g(t_i)}=\lambda(t_i)^{\frac12}h(t_i)dV_{g_i(0)}
\end{align*}
and then
\begin{align*}
    h_{\infty}^2(0)=\lim_{i\to\infty}h_i^2(0)
    =\lim_{i\to\infty}\lambda(t_i)h^2(t_i)
    =\frac12>0,
\end{align*}
which means that $(g_{\infty}(0),H_{\infty}(0))$ is non-trivial.
It is easy to check that the singularity model is the shrinking generalized soliton constructed in Example \ref{eg-S2R}.
\end{proof}

\begin{corollary}
	There exists non-trivial generalized singularity models.
\end{corollary}

The example constructed above demonstrates that the existence of a torsion-bounding function is not a necessary condition for the non-collapsing of the generalized flow.

\begin{corollary}
	There exists non-collapsing generalized Ricci flow which does not admit any torsion-bounding functions.
\end{corollary}
\begin{proof}
	For the case in Proposition \ref{eg-S2S1}, assume there exists a torsion-bounding function $\psi\geq0$ such that
 \begin{align*}
     \left(\pdt-\Delta\right)\psi\leqslant -|H|^2=-6h^2
 \end{align*}
 By the maximum principle, 
 \begin{align*}
     \psi(x,t)\leqslant \psi_0-6\int_0^t h^2(s)ds
     ,\quad
     \psi_0:=\max_{x}\psi(x,0)
 \end{align*}
 However, there exists $t_0<T$ such that
 \begin{align*}
     \psi(x,t)\leqslant\psi_0-6\int_0^{t_0}h^2(s)ds<0
 \end{align*}
 since $\int_0^Th^2(s)ds=\infty$.
 It is a contradiction.
\end{proof}

In fact, we can show that for any shrinking generalized solitons with $H\neq0$, if exists, will never admit any torsion-bounding functions.

\begin{proposition}\label{prop-soliton_cannot_torbounding}
    A compact shrinking generalized soliton admits a torsion-bounding function if and only if it is a Ricci soliton.
\end{proposition}
\begin{proof}
    Assume $(M^n,g,H,X)$ is a shrinking generalized soliton on compact manifold $M^n$ with vector field $X$.
    Let $\phi_t$ be the family of diffeomorphism generated by $X_t:=\frac{1}{1-t}X$. 
    The solution generated by soliton is given by
    $g_t=(1-t)\phi_t^*g$ and $H_t=(1-t)\phi_t^*H$.
    We will show that there exists smooth functions $u>0$ and $\sigma(t)$ such that $u_t=\sigma(t)\phi_t^*u$ satisfies the conjugate heat equation $\square^*u_t=0$.
    By direct computation, we have
\begin{align*}
\pdt\left(\sigma(t)\phi_t^*u\right)=\phi_t^*\left(\sigma'u+\sigma\Li_{X_t}u\right)=\phi_t^*\left(\sigma'u+\frac{\sigma}{1-t}\Li_{X}u\right),
\end{align*}
and
\begin{align*}
\Delta_{g_t}\left(\sigma(t)\phi_t^*u\right)=\phi_t^*\left(\frac{\sigma}{1-t}\Delta_{g} u\right),  
\end{align*}
By definition, we get
\begin{align*}
    \left(R_t-\frac14|H_t|_{g_t}^2\right)\left(\sigma(t)\phi_t^*u\right)=\phi_t^*\left[\frac{\sigma}{1-t}\left(R_g-\frac14|H|_g^2\right)u\right].
\end{align*}
Combining above together, we get
    \begin{align*}
        \square^*u=&\left(-\pdt-\Delta_{g_t}+R_t-\frac14|H_t|_{g_t}^2\right)\left(\sigma(t)\phi_t^*u\right)\\
        =&\frac{\sigma}{1-t}\phi_t^*\left(-\Delta_gu-\Li_Xu+\left(R_g-\frac14|H|_g^2-\frac{\sigma'(1-t)}{\sigma}\right)u\right).
    \end{align*}
    If we take $\sigma(t)=(1-t)^{-n/2}$, then $\frac{\sigma'(1-t)}{\sigma}=\frac{n}{2}$. 
    By tracing the soliton equation, we have 
    \begin{align*}
R-\frac14|H|^2-\frac{n}{2}=-\diver X    
    \end{align*}

Hence the existence of $u_t$ as we claimed is equivalent to the solvability of the next elliptic equation.
    \begin{align*}
        \Delta_g u+\left<X,\nabla u\right>+\diver X \cdot u
        =\diver(\n u+uX)
        =0,
    \end{align*}
By Lemma \ref{lem-divergencepde} in Appendix, there exists a unique strictly positive solution $u$ up to a multiple constant.
And in this case $u_t=(1-t)^{-n/2}\phi_t^*u>0$.

Assume that we have a torsion-bounding function $\psi_t$.
Applying formula \eqref{eq-conjugation-relationship} and noticing $u_t$ is a positive solution to the conjugate heat equation, we have
    \begin{align*}
        \ddt\int_M \psi_t u_t dV=\int_M \square\psi_t\cdot u_tdV_t\leqs-\int_M|H_t|_{g_t}^2u_tdV_t=-\frac{1}{1-t}\int_M|H_0|_{g_0}^2u_0dV_0=:-\frac{a_0}{1-t}.
    \end{align*}
    If $H\neq0$, we have $a_0>0$. Then
    \begin{align*}
        0\leqs\int_M\psi_tu_t dV \leqs \int_M\psi_0 u_0 dV+a_0\log(1-t)\to -\infty \text{ as }t\to 1-.
    \end{align*}
    This leads to a contradiction when $t$ is close to $1$ sufficiently.
\end{proof}

\section{Solitons of pluriclosed flow}\label{sec-pluriclosedsoliton}
Given a one parameter family of Hermitian metric $\omega(t)$ on a complex manifold $M^{2n}$. We say $(M^{2n},J,\omega(t))$ is a pluriclosed flow if it satisfies
\begin{align*}
    \pdt\omega=-\rho^{1,1},
\end{align*}
where $\rho^{1,1}$ is the $(1,1)$ part of the Bismut Ricci curvature. Bismut connection is a natural choice of connections in complex geometry, which is the unique Hermitian connection compatible with the metric and having totally skew-symmetric torsion.

\begin{proposition}[\cite{MR3110582}, Proposition 6.3, Proposition 6.4]
    Given $(M^{2n},J,\omega(t))$ a pluriclosed flow, one has
    \begin{align*}
        \pdt g&=-\Ric_g+\frac14H^2-\frac12\Li_{\theta^\sharp}g,\\
        \pdt H&=-\frac12\Delta_d H-\frac12\Li_{\theta^\sharp}H,
    \end{align*}
    where $H=d^c\omega\in \Lambda^3(M)$,  $\theta=d^*\omega\circ J$ is the Lee form.
\end{proposition}

Thus the pluriclosed flow is a special case of generalized Ricci flow module gauge transformation. Moreover, we show that under the gauge transformation, the heat equation will also behave well.

\begin{proposition}
    Given $(M^{2n},J,\omega(t))$ a pluriclosed flow. Let $\phi_t$ be the family of diffeomorphism generated by $X_t:=\theta_{2t}^{\sharp}$, $\hat{g}_t:=\phi_t^*g_{2t}$, $\hat{H}_t:=\phi_t^*H_{2t}$. Then $(M^{2n},\hat{g}_t,\hat{H}_t)$ solves the generalized Ricci flow. Moreover, given $f_t\in C^\infty(M)$, $h_t\in C^\infty(M)$ satisfying
    \begin{align*}
        \pdt f=\Delta^C_{g_t} f+h,
    \end{align*}
    Then for $\hat{f}_t:=\phi_t^* f_{2t}$, $\hat{h}_t:=\phi_t^* h_{2t}$, we have
    \begin{align*}
        \pdt \hat{f}=\Delta^{LC}_{\hat{g}_t} \hat{f}+2\hat{h},
    \end{align*}
    where $\Delta^C_g$ is the Chern Laplacian w.r.t $g$, and $\Delta^{LC}_{\hat{g}}$ is the Levi-Civita Laplacian w.r.t $\hat{g}$.
\end{proposition}
\begin{proof}By straight computation,
    \begin{align*}
        \pdt\hat{g}_t=&\phi_t^*\left(\pdt g_{2t}+\Li_{\theta_{2t}^\sharp}g_{2t}\right)=\phi_t^*\left(-2\Ric_{g_{2t}}+\frac12H^2_{2t}\right)=-2\Ric_{\hat{g}_t}+\frac{1}{2}\hat{H}_t^2,\\
        \pdt\hat{H}_t=&\phi_t^*\left(\pdt H_{2t}+\Li_{\theta_{2t}^\sharp}H_{2t}\right)=\phi_t^*\left(-\Delta_dH_{2t}\right)=-\Delta_d\hat{H}_t.
    \end{align*}
    \begin{align*}
        \pdt\hat{f}_t=&\phi_t^*\left(\pdt f_{2t}+\Li_{\theta_{2t}^\sharp}f_{2t}\right)=\phi_t^*\left(2\Delta^C_{2t}f_{2t}+2h_{2t}+\left<\theta_{2t}^\sharp,\n f_{2t}\right>\right)\\
        =&\phi_t^*\left(\Delta^{LC}_{2t}f_{2t}+2h_{2t}\right)=\Delta^{LC}_{\hat{g}_t}\hat{f}_t+2\hat{h}_t.
    \end{align*}
    The last line follows from the fact that $\Delta^{LC}u=2\Delta^C u+\left<\theta^\sharp,\n u\right>$ by the computation in \cite[Proof of Proposition 4.1]{MR2673720}.
\end{proof}

Next we introduce solitons of the pluriclosed flow.

\begin{definition}
    $(M^{2n},J,\omega,X)$ is called a \emph{pluriclosed soliton} if $X$ is the real part of a holomorphic vector field and there exists a constant $\lambda$ such that
    \begin{align*}
        \rho^{1,1}=\lambda\omega-\frac{1}{2}\Li_{X}\omega.
    \end{align*}
    A pluriclosed soliton is called gradient if $X=-\theta^\sharp+\nabla f$ for some $f\in C^\infty(M)$.
\end{definition}
\begin{remark}
    The Lee form term in the definition of gradient soliton is due to the gauge transformation between pluriclosed flow and generalized Ricci flow. In the \kah \ case, $\theta=0$, the definition coincides with that in \kah-Ricci flow. If we define the gradient soliton as $X=\nabla f$, then the steady or expanding gradient pluriclosed solitons must be \kah\ by Proposition \ref{nosteadyYesoliton}. However, it's unknown for the shrinking case.
\end{remark}
\begin{question}
    Does there exist non\kah\ shrinking pluriclosed soliton with $X=\nabla f$?
\end{question}

Before proving Proposition \ref{nosteadyYesoliton}, we give a lemma to show the relationship between pluriclosed and balanced condition.

\begin{lemma}\label{balancepluriclosed}
    Suppose $(M^{2n},J,\omega)$ be a compact pluriclosed manifold, if there exists $f\in C^\infty(M)$ such that $\hat{\omega}:=e^f\omega$ is balanced, i.e. $d^*_{\hat{\omega}}\hat{\omega}=0$, then $\omega$ and $\hat{\omega}$ are both \kah.
\end{lemma}
\begin{proof}
   By \cite[(2.13)]{MR1836272}, there holds a pointwise equality:
    \begin{align*}
        \left<dd^c\omega,\omega\wedge\omega\right>=2|\theta|^2-2|d \omega|^2+2d^*\theta.
    \end{align*}
    $\hat{\omega}$ is balanced, which is equivalent to $\hat{\theta}=0$ since $\hat{d}^*\hat{\omega}=-\hat{\theta}\circ J=0$. By the conformal transformation of Lee form $\theta_{e^fg}=\theta_g+(n-1)df$, we have $\theta_g=du$ is exact, where $u=-(n-1)f$. Then we have
     \begin{align*}
        -\Delta u+|du|^2=|d \omega|^2\geqs0.
    \end{align*}
    Then $d\omega=0$ and $u$ is constant by strong maximum principle.
\end{proof}
\begin{remark}
    The Fino-Vezzoni conjecture \cite{MR3327047} asks whether a compact complex manifold that admits both pluriclosed metric and balanced metric must be \kah. This lemma shows that a pluriclosed metric and a balanced metric cannot lie in the same conformal class unless they are \kah.
\end{remark}

\begin{proposition}\label{nosteadyYesoliton}
    Suppose $(M^{2n},J,\omega,X)$ be a steady or expanding pluriclosed soliton, and there exists $f\in C^\infty(M,\rean)$ such that $X=\nabla f$. Then $H\equiv0$.
\end{proposition}
\begin{proof}
    Let $\phi_t$ be a family of biholomorphism generated by $(1-\lambda t)^{-1}X^{1,0}$, then $\omega_t=(1-\lambda t)\phi_t^*\omega_0$ solves the pluriclosed flow. Since $H_t:=d^c \omega_t=(1-\lambda t)\phi_t^*d^c\omega_0=(1-\lambda t)\phi_t^*H_0$, $(g_t,H_t)$ evolves only up to scaling and biholomorphism. Let $\psi_t$ be the family of diffeomorphism generated by $\theta^\sharp$, then $(\tilde{g}_t,\tilde{H}_t)=(\psi_t^*g_{2t},\psi_t^*H_{2t})$ solves the generalized Ricci flow. Since $(\tilde{g}_t,\tilde{H}_t)=(\Psi_t^*g_0,\Psi_t^*H_0)$ for $\Psi_t=\phi_{2t}\circ\psi_t$, $(g,H,\nabla f+\theta^\sharp)$ is a generalized steady soliton. By the $\mathcal{F}$-entropy and expanding entropy $\W_+$ \cite{MR4284898}, the compact generalized steady or expanding solitons are always gradient, then there exists $u\in C^\infty(M)$ such that $\nabla f+\theta^\sharp=\nabla (u+f)$, i.e. $\theta=du$ is exact. Then $\omega$ is both conformally balanced and pluriclosed, hence $\omega$ is \kah\ by Lemma \ref{balancepluriclosed}.
\end{proof}

For the gradient shrinking pluriclosed soliton of our definition, we can apply the Theorem \ref{noshrinkingsoliton}.
\begin{proposition}
    Suppose $(M^{2n},J,\omega,f)$ is a gradient shrinking pluriclosed soliton on a compact complex manifold, then $\omega$ is K\"{a}hler.
\end{proposition}
\begin{proof}
    By the argument in the proof of Proposition \ref{nosteadyYesoliton}, $(g,H,X+\theta^\sharp=\nabla f)$ is a gradient shrinking generalized soliton, where $H=d^c\omega$. Then by Theorem \ref{noshrinkingsoliton}, $H=0$.
\end{proof}
\begin{remark}
    Streets \cite{MR4023384} gives a proof in the complex surfaces case using classification theory of compact complex surfaces.
\end{remark}

Next under a cohomological assumption, we can show the pluriclosed soliton of any type is \kah ian.

\begin{theorem}\label{thm-generalizedsoliton}
    Suppose $(M^{2n},J,\omega,X)$ is a pluriclosed soliton on a compact complex manifold with $[\dd\omega]=0\in H^{2,1}_{\db}(M)$. Then $\omega$ is a gradient K\"{a}hler-Ricci soliton.
\end{theorem}
\begin{proof}
   Since $[\dd\omega]=0$ there exists $\phi\in \mathcal{A}^{2,0}$ such that $\dd\omega+\db\phi=0$. We consider the normalized flow
    \begin{align*}
        \pdt\omega=-\rho^{1,1}+\lambda \omega,\\
        \pdt\phi=-\rho^{2,0}+\lambda \phi,
    \end{align*}
    where $\lambda\in\{-1,0,1\}$ depending the type of the soliton. Thus the soliton becomes the normalized steady soliton. Note that the condition $\dd\omega+\db\phi=0$ is preserved under the normalized flow by \cite{Ye2022PluriclosedFA}. By \cite[Proposition 3.23]{MR4629758}, we have $R^B_{ijk\bar{l}}=\nabla^C_kT_{ij\bar{l}}$ where $R^B$ is the Bismut curvature and $T=\dd\omega$, then
    \begin{align*}
        -\rho^{2,0}_{ij}=-g^{\bar{l}k}R^B_{ijk\bar{l}}=-g^{\bar{l}k}\nabla^C_kT_{ij\bar{l}}=g^{\bar{l}k}\nabla^C_k(\db\phi)_{ij\bar{l}}=g^{\bar{l}k}\nabla^C_k\nabla_{\bar{l}}\phi_{ij}=\Delta^C\phi_{ij}.
    \end{align*}
    By a normalized modification of \cite[Lemma 4.7]{MR3808262}, we have
    \begin{equation}\label{evolvephi}
        \begin{aligned}
            \left(\pdt-\Delta\right)|\phi|^2=&-|\nabla\phi|^2-|\overline{\nabla}\phi|^2-2\left<Q+\lambda\omega,\tr_g(\phi\otimes\bar{\phi})\right>+2\left<\phi,\lambda\phi\right>\\
        =&-|\nabla\phi|^2-|\overline{\nabla}\phi|^2-2\left<Q,\tr_g(\phi\otimes\bar{\phi})\right>\\
        \leqs&-|T|^2.
        \end{aligned}
    \end{equation}
    The last inequality follows from $\overline{\nabla}\phi=\db\phi=-\dd\omega=-T$ and $Q\geqs0$ is semidefinite.

    If $\phi$ evolves only up to diffeomorphism, then we have $T=0$ by the strong maximum principle. However, it's not clear whether it holds since $\phi$ satisfying the condition is not unique, which may differs up to a $\db$-closed $(2,0)$-form. We will choose a special $\phi$ to finish the proof.
    
    Define
    \begin{align*}
        \Psi_t:=\{\phi\in\mathcal{A}^{2,0}:\dd\omega_t+\db\phi=0\}=\hat{\phi}_t+\{\phi\in\mathcal{A}^{2,0}:\db\phi=0\}.
    \end{align*}
    Note that for $\phi\in\mathcal{A}^{2,0}$, $\db^*\phi=0$. So $\Psi_t=\hat{\phi}_t+\mathcal{H}^{2,0}_{\db}$ is a finite dimensional affine space. Let $f_t$ be the family of biholomorphism generated by the soliton vector field, $\omega_t=f_t^*\omega_0$. Then $f_t^*:\Psi_0\to\Psi_t$ is a isomorphism since $\dd\omega_t+\db f_t^*\phi_0=f_t^*(\dd\omega_0+\db\phi_0)$. Define
    \begin{align*}
        h_t:\Psi_t\to\rean:\quad \phi\mapsto \max_{x\in M}|\phi|^2_{g_t}(x)
    \end{align*}
    We claim there exists $\tilde{\phi}_t\in\Psi_t$ such that $h_t(\tilde{\phi}_t)=\inf_{\phi\in\Psi_t}h_t(\phi)$. Since $h_t$ is continuous, it suffices to show $h_t$ is proper, i.e. $\lim_{\alpha\to\infty}h_t(\hat{\phi}_t+\sum_{i=1}^m\alpha_i e_i)=+\infty$, where $\{e_i\}$ is the $L^2$-orthogonal basis of $\mathcal{H}^{2,0}_{\db}$. We compute the $L^2$ norm to get the properness.
    \begin{align*}
        h_t \geqs\fint_M\left|\hat{\phi}_t+\sum_{i=1}^m\alpha_i e_i\right|^2 dV=\fint_M \left(\sum_{i=1}^m|\alpha_i|^2 |e_i|^2+2\sum_{i=1}^m|\alpha_ie_i\hat{\phi}_t|+|\hat{\phi}_t|^2\right) dV\to+\infty 
    \end{align*}
    Now we use $\tilde{\phi}_0$ to be the initial data of the normalized flow. By (\ref{evolvephi}), $h_t(\phi_t)\leqs h_0(\phi_0)$. On the other hand, since $h_t(f_t^*\phi)=h_0(\phi)$, we have $h_0(\phi_0)=\min h_0(\Psi_0)=\min h_t(\Psi_t)\leqs h_t(\phi_t)$. So $h_t(\psi_t)$ is constant along the flow. By the strong maximum principle, $T=0$.

\end{proof}
\begin{remark}
    This generalizes Streets' theorem on K\"{a}hler manifolds \cite[Proposition 3.5]{MR4023384}. Note that for the manifolds which $\dd\db$-lemma holds, the assumption of our result always hold. By the pluriclosed condition, $\dd\omega$ is closed and $\dd$-exact, hence is $\dd\db$-exact, i.e. there exists $\alpha\in \mathcal{A}^{1,0}$ such that $\dd\omega=\dd\db\alpha=-\db\dd\alpha$. Then $[\dd\omega]=0\in H^{2,1}_{\db}$. See \cite{MR3032326} for non\kah\ manifolds that $\dd\db$-lemma holds.
\end{remark}

\section{Appendix}
To complete the proof of Lemma \ref{lem-point-monotonicity}, we need the next lemma.
\begin{lemma}[\cite{MR4284898}, Lemma 3.19]\label{lem-appendix}
    Given $(M^n,g)$ a Riemannian manifold and $H\in \Lambda^3 T^*$, $dH=0$, one has
    \begin{align*}
        (\diver H^2)_i=\frac16\n_i|H|^2-(d^*H)^{mn}H_{imn}.
    \end{align*}
\end{lemma}

Now we list the proof of Lemma \ref{lem-point-monotonicity}.
\begin{proof}[Proof of Lemma \ref{lem-point-monotonicity}]
By definition,
\begin{align*}
    \square^*u=&(-\pdt-\Delta+R-\frac14|H|^2)(4\pi\tau)^{-\frac{n}{2}}e
^{-f}\\
=&\left(-\frac{n}{2\tau}+\pdt f+\Delta f-|\n f|^2+R-\frac14|H|^2\right)u,
\end{align*}
so $u_t$ being a solution to the conjugate heat equation is equivalent to
\begin{align*}
    \pdt f=-\Delta f+|\nabla f|^2-R+\frac14|H|^2+\frac{n}{2\tau}.
\end{align*}
By direct computation, we have
 \begin{align*}
  \left( \pdt + \Delta \right) 2 \Delta f =&\ 2 \left[ \left<2 \Ric - \frac{1}{2} H^2,
\nabla^2 f \right> + \Delta \left( - \Delta f + \brs{\nabla f}^2 - R + \frac{1}{4} \brs{H}^2
\right) \right. \\
&\ \left. \qquad - \left< \frac{1}{2} \diver H^2 - \frac{1}{4} \nabla \brs{H}^2, \nabla f
\right> + \Delta \Delta f \right]\\
=&\ \left< 4 \Ric - H^2, \nabla^2 f \right> + 2 \Delta \brs{\nabla f}^2 - 2 \Delta R +
\frac{1}{2} \Delta \brs{H}^2 \\
&\ \ \qquad + \left< \frac{1}{2} \nabla \brs{H}^2 - \diver H^2, \nabla f
\right>.
 \end{align*}
and
\begin{align*}
 \left( \pdt + \Delta \right) \left(- \brs{\nabla f}^2 \right) =&\ - \left[ 2 \left< \nabla
(- \Delta f + \brs{\nabla f}^2 - R + \frac{1}{4} \brs{H}^2), \nabla f \right> \right.\\
&\ \left. \qquad + \left< 2
\Ric - \frac{1}{2} H^2, \nabla f \otimes \nabla f \right> + \Delta \brs{\nabla f}^2 \right]\\
=&\ 2 \left< \nabla\left(\Delta f - \brs{\nabla f}^2 + R - \frac{1}{4} \brs{H}^2 \right),
\nabla f \right>\\
&\ \qquad + \left< \frac{1}{2} H^2 - 2 \Ric, \nabla f \otimes \nabla f \right> - \Delta
\brs{\nabla f}^2.
\end{align*}
Notice that
\begin{align*}
 \left( \pdt + \Delta \right) R =&\ 2 \Delta R + 2 \brs{\Ric}^2 - \frac{1}{2} \Delta
\brs{H}^2 + \frac{1}{2} \diver \diver H^2 - \frac{1}{2} \left< \Ric, H^2 \right>.
\end{align*}
and
\begin{align*}
 \left( \pdt + \Delta \right) \left( - \frac{1}{12} \brs{H}^2 \right) =&\ -
\frac{1}{12} \left[ \left<6 \Ric - \frac{3}{2} H^2, H^2 \right> + 2 \left< \Delta_d
H, H \right> + \Delta \brs{H}^2 \right]\\
=&\ \left< \frac{1}{8} H^2 - \frac{1}{2} \Ric, H^2 \right> - \frac{1}{6} \left<
\Delta_d H, H \right> - \frac{1}{12} \Delta \brs{H}^2.
\end{align*}
Let $V = 2\Delta f - |\nabla f|^2 + R - \frac{1}{12} |H|^2$.
We obtain
\begin{align*}
 \left( \pdt + \Delta \right) V =&\ \Delta \brs{\nabla f}^2 - \frac{1}{12} \Delta \brs{H}^2 -
\frac{1}{6} \left< \Delta_d H, H \right> + \frac{1}{2} \diver \diver H^2\\
&\ + \left< \frac{1}{2} \nabla \brs{H}^2 - \diver H^2, \nabla f \right> + \left<
\frac{1}{2} H^2 - 2 \Ric, \nabla f \otimes \nabla f \right>\\
&\ + 2\brs{\Ric - \frac{1}{4} H^2 + \nabla^2 f}^2 - 2\brs{\nabla^2 f}^2\\
&\ + 2 \left< \nabla\left( \Delta f - \brs{\nabla f}^2 + R - \frac{1}{4} \brs{H}^2 \right),
\nabla f \right>.
\end{align*}
Using Lemma \ref{lem-appendix} one has
\begin{align*}
 \frac{1}{2} \diver \diver H^2 - \frac{1}{6} \left< \Delta_d H, H \right> -
\frac{1}{12} \Delta \brs{H}^2 = \frac{1}{2} \brs{d^* H}^2,
\end{align*}
and then
\begin{align*}
 \left< \frac{1}{2} \nabla \brs{H}^2 - \diver H^2, \nabla f \right> = \frac{1}{3} \left<
\nabla \brs{H}^2, \nabla f \right> + \left< d^* H, i_{\nabla f} H \right>.
\end{align*}
Also one has
\begin{align*}
 \Delta \brs{\nabla f}^2 - 2 \brs{\nabla^2 f}^2 - 2 \left< \Ric, \nabla f \otimes \nabla f \right> =
2 \left<\nabla \Delta f, \nabla f \right>.
\end{align*}
Therefore
\begin{align*}
 \left( \pdt + \Delta \right) V =&\ 2 \brs{\Ric - \frac{1}{4} H^2 + \nabla^2 f}^2 +
\frac{1}{2} \brs{d^* H + i_{\nabla f}  H}^2 + 2 \left< \nabla V, \nabla f \right>.
\end{align*}
 Let $W=(T-t)V+f-n$. We have
 \begin{align*}
      \left( \pdt + \Delta \right)W=&-V+2\tau \brs{\Ric - \frac{1}{4} H^2 + \nabla^2 f}^2 +
\frac{\tau}{2} \brs{d^* H + i_{\nabla f}  H}^2\\
&+ 2\tau \left< \nabla V, \nabla f \right>+|\nabla f|^2-R+\frac14|H|^2+\frac{n}{2\tau}\\
=&-2
\Delta f  - 2R + \frac{1}{3} |H|^2+\frac{n}{2\tau}+ 2 \left< \nabla W, \nabla f \right>\\
&+2\tau \brs{\Ric - \frac{1}{4} H^2 + \nabla^2 f}^2 +
\frac{\tau}{2} \brs{d^* H + i_{\nabla f}  H}^2\\
=&2\tau \brs{\Ric - \frac{1}{4} H^2 + \nabla^2 f-\frac{g}{2\tau}}^2 +
\frac{\tau}{2} \brs{d^* H + i_{\nabla f}  H}^2- \frac{1}{6} |H|^2+ 2 \left< \nabla W, \nabla f \right>.
 \end{align*}
Finally,
\begin{align*}
\left(\pdt + \Delta \right) v =&\ \left( \pdt + \Delta \right) \left(W u \right)\\
=&\ \left(\left( \pdt + \Delta \right) W \right) u + W \left(\pdt + \Delta \right) u + 2
\left<\nabla W, \nabla u \right>\\
=&\ \left( 2 \tau \brs{\Ric - \frac{1}{4} H^2 + \nabla^2 f - \frac{g}{2 \tau}}^2 +
\frac{\tau}{2} \brs{d^* H + i_{\nabla f} H}^2 \right.\\
&\ \left. - \frac{1}{6} \brs{H}^2 + 2 \left< \nabla W, \nabla f \right> \right) u +
W \left( R u - \frac{1}{4} \brs{H}^2 u \right) - 2 \left< \nabla W, \nabla f
\right> u\\
=&\ 2 \tau \brs{\Ric - \frac{1}{4} H^2 + \nabla^2 f - \frac{g}{2 \tau}}^2 u + \frac{\tau}{2}
\brs{d^* H + i_{\nabla f} H}^2 u - \frac{1}{6} \brs{H}^2 u\\
&\ + R v - \frac{1}{4} \brs{H}^2 v.
\end{align*}
Hence
\begin{align*}
    \square^*v=-\left(2 \tau \brs{\Ric - \frac{1}{4} H^2 + \nabla^2 f - \frac{g}{2 \tau}}^2 + \frac{\tau}{2}
\brs{d^* H + i_{\nabla f} H}^2  - \frac{1}{6} \brs{H}^2 \right)u.
\end{align*}
\end{proof}

\begin{lemma}[Gauduchon\cite{MR0742896}]\label{lem-divergencepde}
Given a smooth vector field $X$ on a compact Riemannian manifold $(M^n,g)$.
Then kernel space of the linear operator $L$ defined by
\begin{align*}
Lv=\Delta v+\diver(vX)
\end{align*}
is $\ker(L)=\{\lambda u|\lambda\in\mathbb{R}\}$, in which $u$ is strictly positive.
\end{lemma}
\begin{proof}
    Consider the $L^2$ adjoint $L^*$ of the operator $L$, which has no zero order term:
    \begin{align*}
        L^*v=\Delta v-\left<X,\n v\right>.
    \end{align*}
    Then by the strong maximum principle, $L^*v=0$ if and only if $v$ is constant, i.e. $\dim \ker(L^*)=1$.

    Next we apply the standard argument in elliptic theory. Take $\sigma>0$ sufficiently large such that $L^*_{\sigma}:=L^*-\sigma I:W^{1,2}\to\left(W^{1,2}\right)^*$ is invertible. Then $(L^*_\sigma)^{-1}:W^{1,2}\to (W^{1,2})^*\to W^{1,2}$ is a compact operator. So $L^*=L^*_\sigma+\sigma I=L^*_\sigma\circ (I+\sigma (L^*_\sigma)^{-1}\circ I)$ is a composition of a Fredholm operator with a invertible operator.

    By the Fredholm theory, we have $\dim\ker (L)=\dim\ker (L^*)=1$ and $\ker (L)=\left(\image L^*\right)^\perp$.
    Then $\ker(L)=\{\lambda u|\lambda\in\mathbb{R}\}$, where $u\neq0$. If there exists $x_1,x_2\in M$ such that $u(x_1)<0$, $u(x_2)>0$, then we can take a function $\phi>0$ such that $\int_M u\phi dV=0$. By the strong maximum principle, $\phi\notin \image L^*$. Since $W^{1,2}(M)=\ker(L)\oplus \image L^*$, we have a decomposition $\phi=\lambda_0u+f$, where $\lambda_0\neq0$ and $f\in \image L^*$. Then
    \begin{align*}
        0=\int_M u\phi dV=\int_M u(\lambda_0u+f) dV=\lambda_0\int_M u^2 dV\neq 0.
    \end{align*}
    It's a contradiction, so we can assume $u\geqs0$. By the strong maximum principle, we have $u>0$.
\end{proof}

\bibliographystyle{ytamsalpha}
\bibliography{ref.bib}	
	
\end{document}